\documentclass[12pt,a4paper,oneside]{amsart}

\usepackage{subcaption}

\usepackage[margin=1in]{geometry}
\usepackage{amsfonts,amsmath,amssymb,amsthm,enumitem}
\usepackage{mathtools}
\usepackage{graphicx}
\usepackage{doi}
\usepackage{tikz}
\usetikzlibrary{quotes,angles,positioning, shapes.geometric, arrows.meta, calc}
\usepackage{hyperref}
\hypersetup{hypertexnames=false}

\numberwithin{equation}{section}

\newtheorem{theorem}{Theorem}[section]
\newtheorem{lemma}[theorem]{Lemma}

\newtheorem{conjecture}[theorem]{Conjecture}

\newcommand{\R}{\mathbb{R}}
\newcommand{\PP}{\mathbb{P}}
\newcommand{\EE}{\mathbb{E}}
\renewcommand{\epsilon}{\varepsilon}

\newcommand{\cube}{[0,1]^d}
\newcommand{\cC}{\mathcal{C}}

\usepackage{calc}
\usepackage{ragged2e}
\usepackage{setspace}
\newcommand{\FundingLogos}{%
  \raisebox{0pt}{\includegraphics[height=1.5cm]{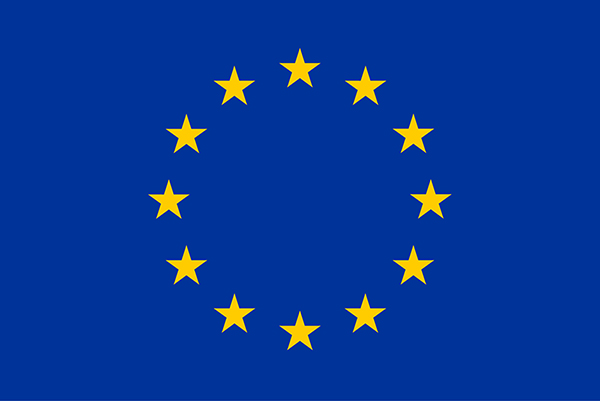}}%
  \hspace{1em}%
  \raisebox{0pt}{\includegraphics[height=1.5cm]{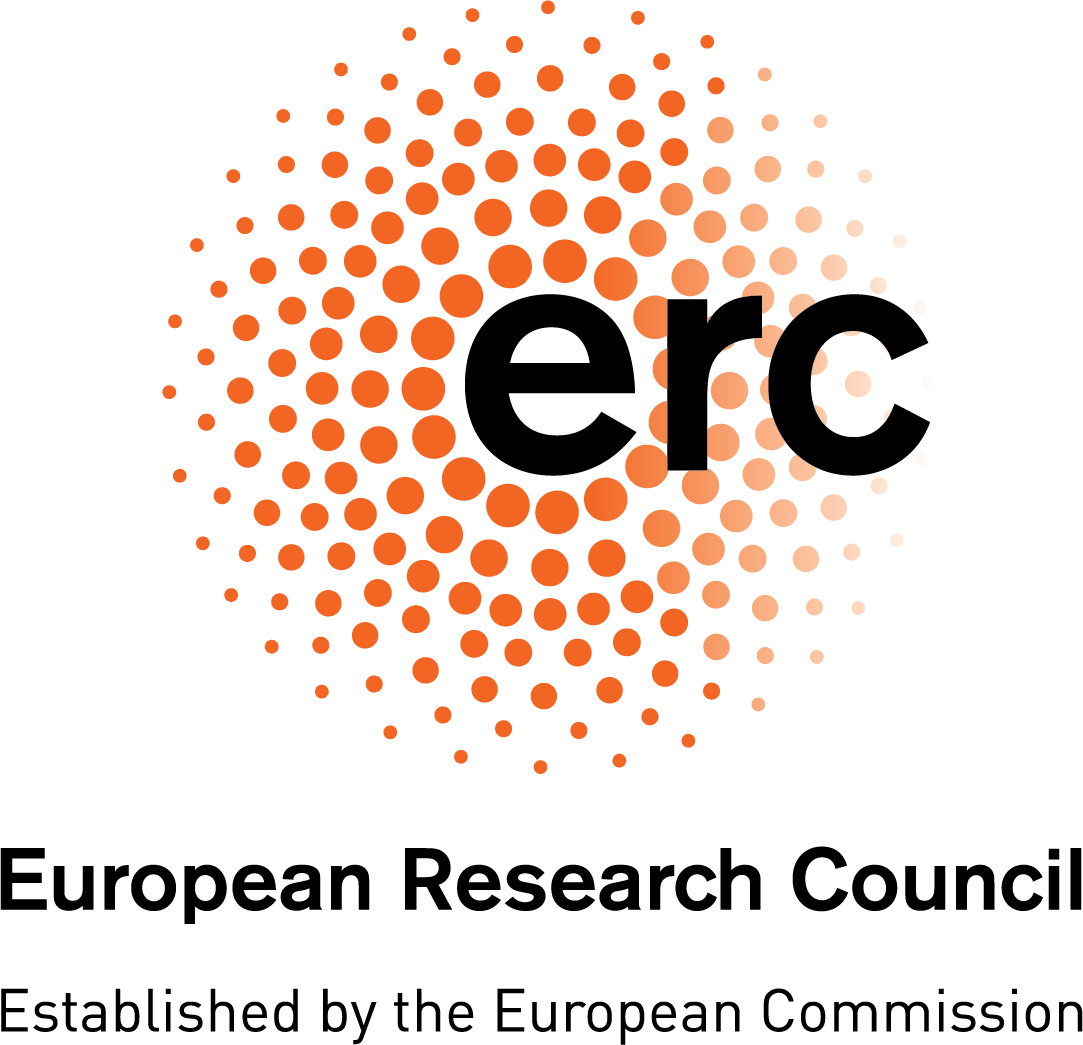}}%
}

\usepackage{tcolorbox}
\tcbset{colback=white,boxrule=0.5pt,arc=2pt,left=4pt,right=4pt,top=4pt,bottom=4pt}

\begin{document}

\title[Concentration for random Euclidean combinatorial optimization]{Concentration for random Euclidean combinatorial optimization}

\author{
Matteo D'Achille
\address[Matteo D'Achille]{Institut \'Elie Cartan de Lorraine, CNRS, Universit\'e de Lorraine, F-57070, Metz, France}
\email{matteo.d-achille@univ-lorraine.fr}
\hspace*{0.5cm}
Francesco Mattesini
\address[Francesco Mattesini]{Department of Mathematics, CIT School, Technical University of Munich \& Munich Center for Machine Learning (MCML), Munich, Germany}
\email{francesco.mattesini@tum.de}
\hspace*{0.5cm}
Dario Trevisan
\address[Dario Trevisan]{Universit\`a di Pisa, Italy}
\email{dario.trevisan@unipi.it}
}

\begin{abstract}
We prove concentration bounds for random Euclidean combinatorial optimization
problems with $p$--costs.  For bipartite matching in dimension $d\ge3$, our
new result concerns the range $d\le p<d^2/2$, extending earlier concentration
results for $p<d$.  The argument combines a tensorized Poincar\'e inequality
with a deterministic local edge-to-energy estimate yielding power-law control
of maximal optimal edges.  We give analogous complete arguments for mono- and
bipartite Euclidean traveling salesperson problems, and formulate a
conjectural $p\!\to\!q$ transfer principle whose validity would remove the
present restriction on $p$.
\end{abstract}

\maketitle

\section{Introduction and overview}
\label{sec:intro}

We study concentration properties for random Euclidean optimization problems built on i.i.d.\ uniform point clouds in the cube $\cube$ (with Euclidean distance).
Working in $\cube$ keeps the presentation simple.  We expect analogous
arguments to extend to sufficiently regular domains or compact manifolds,
provided one has the required local volume bounds, geodesic geometry, and
Poincar\'e inequality.

Two prototypical examples are:
\begin{itemize}[leftmargin=2.5em]
\item 
{
(Random Euclidean) Bipartite matching.\footnote{
{
We use the term \emph{random Euclidean bipartite matching}, in line with the probabilistic literature on Euclidean models \cite{Sicuro17thesis,GoldTrev24ran}. Note that in statistical physics and combinatorics the more common terminology is \emph{random assignment problem}, see e.g.\ \cite{MezardParisi88,Aldous01}, for the same model with general (not necessarily Euclidean) i.i.d.\ costs.} 
}
Two independent point clouds $\mathbf X=(X_i)_{i=1}^n$ and $\mathbf Y=(Y_j)_{j=1}^n$ are matched by a permutation $\sigma$, minimising the sum of $p$--powers of distances, where $p\ge 1$ is a parameter.
}
\item Euclidean traveling salesperson problem (TSP). A Hamiltonian cycle is chosen on one cloud (monopartite TSP), or on two clouds with an alternation constraint (bipartite TSP), minimising the sum of $p$--powers of edge lengths.
\end{itemize}

In all these models, the disorder is geometric (the random points) and the
optimiser selects a low-energy combinatorial structure.  In the dimensions
considered below, the natural energy scale is $n^{1-p/d}$, and our goal is to
quantify sample-to-sample fluctuations relative to that scale.

\subsection{Main results and proof strategy (informal)}
Our results show that, in dimension $d\ge 3$ and for $1\le p<d^2/2$, the optimal costs for matching and TSP are \emph{self-averaging} at the energy scale: with polynomially high probability,
\[
|\, \cC -\EE[\cC]\,| \ll n^{1-p/d},
\]
where $\cC$ denotes the optimal cost. Precise statements are given in Theorems~\ref{thm:conc_matching}, \ref{thm:conc_tsp}, and \ref{thm:conc_bTSP}. The argument consists of two main layers.

For bipartite matching, the range $p<d$ is essentially covered by the
general framework of \cite[Proposition~5.3]{GoldTrev24ran}.  The new content
of Theorem~\ref{thm:conc_matching} is therefore {the half-open interval}
\[
d\le p<d^2/2,
\]
obtained through the maximal-edge estimate developed below.

\paragraph{(A) \emph{Analytic concentration tool: Poincar\'e inequality.} }
Let $\mu$ be the uniform probability measure on $(\cube)^n$ (monopartite problems) or on $(\cube)^{2n}$ (two-cloud problems).
The tensorized Poincar\'e inequality gives, for every locally Lipschitz $F$,
\begin{equation}\label{eq:poincare_intro}
\mathrm{Var}_\mu(F)\le C_{\mathrm P}\,\EE_\mu\big[|\nabla F|^2\big].
\end{equation}
Applying \eqref{eq:poincare_intro} to the optimal cost  $\cC$ reduces concentration to bounding $|\nabla \cC|$.
At Lebesgue-a.e.~configuration of the input points, the optimal-cost
functional, viewed as a locally Lipschitz function on $(\cube)^n$ or
$(\cube)^{2n}$, is differentiable by Rademacher's theorem{, see e.g.~\cite[Theorem 10.8]{VilOandN}}.  At such points one
has, when $p\ge2$, gradient bounds of the form
\begin{equation}\label{eq:grad_scheme_intro}
|\nabla \cC |^2 \;\lesssim\; \sum_e |e|^{2(p-1)}
\;\le\; \cC \cdot \big(\sup_e |e|\big)^{p-2},
\end{equation}
where the sum and supremum are over the edges selected by an optimiser.

\paragraph{(B) \emph{ Geometric input: excluding long edges by local stability.} }
The key difficulty is to control the maximal edge length $\sup_e |e|$.
The mechanism is:
\begin{enumerate}[leftmargin=2.5em]
\item a \emph{mesoscopic diffuseness} event, ensuring that every ball of radius $\gtrsim n^{-\alpha/d}$ contains many points (with very high probability);
\item a \emph{local optimality} condition: for matching, a two-point swap
exchanges the two assigned partners; for TSP, a $2$--opt move removes two tour
edges and reconnects the resulting paths to obtain another Hamiltonian cycle;
\item a deterministic lemma converting $2$--opt into a \emph{local edge-to-energy} inequality (that we prove in detail for bipartite matching; see Lemma~\ref{lem:l_infty_to_loc}).
\end{enumerate}
Combining these ingredients shows that, for every
$\alpha<p/(p+d)$,
\[
\sup_e |e|\lesssim_\alpha n^{-\alpha/d}
\]
with arbitrarily high polynomial probability.  Equivalently, the endpoint
exponent is $n^{-p/(d(p+d))+o(1)}$.  This closes the Poincar\'e estimate.
The exponent $p<d^2/2$ is exactly the condition that the resulting polynomial tail exponent is positive.

\subsection{A conjectural \texorpdfstring{$p\!\to\!q$}{p → q} transfer principle}
The exploratory simulations in Section~\ref{sec:numerics} do not test the
fluctuation exponent, but they motivate asking whether the present range
$p<d^2/2$ is intrinsic. A natural strengthening would be to control not only
the maximal edge, but higher moments of the edge lengths under the
$p$--optimiser.

\begin{conjecture}[$p\!\to\!q$ transfer for the $p$--optimal matching]\label{conj:ptoq}
Fix $d\ge 3$ and $p\ge 1$. For every $q>p$ there exists $c=c(d,p,q)<\infty$ such that for every
$p$--optimal matching $\sigma_p$,
\[
\EE\Big[\sum_{i=1}^n |X_i-Y_{\sigma_p(i)}|^{q}\Big]
\;\le\; c\, n^{1-q/d}.
\]
\end{conjecture}

Conjecture~\ref{conj:ptoq} matches the heuristic that the $p$--optimiser still pairs points at the typical matching scale $n^{-1/d}$, hence has $q$--energy of order $n^{1-q/d}$ for every $q>p$. We note that the case $q<p$ is trivially true by optimality of $\sigma_p$ and an application of H\"older's inequality with exponent $p/q$. If the conjecture were true (with $q=2p-2$), it would remove the restriction $p<d^2/2$ in our concentration bounds. Similar conjectures can be stated for TSP and bipartite TSP cases.

\subsection{Related literature and context}\label{subsec:literature}

Random Euclidean combinatorial optimization lies at the intersection of probability,
geometric measure theory, statistical physics, and algorithms.
The foundational results are the law of large numbers for Euclidean functionals,
starting from the Beardwood–Halton–Hammersley theorem and its extensions to broad
classes of subadditive Euclidean functionals; see also the monographs by Steele and
Yukich for a unified probabilistic framework covering the minimum spanning tree (MST), matching, TSP, and related models \cite{BHH59,Steele97,Yukich98}.

On the probabilistic side, quantitative asymptotics and fluctuation theory for Euclidean
matching and assignment problems originate in the work of Ajtai–Komlós–Tusnády and its
subsequent refinements via multiscale, Fourier, and empirical-process techniques \cite{AKT84,BoLe21}.
Modern developments connect these ideas to geometric probability tools such as
uniform occupancy, stabilization, and VC-type arguments
\cite{PenroseRGG03,DevroyeLugosi01}.

A complementary viewpoint arises from statistical physics and optimal transport, 
{
where explicit descriptions of the cost where firstly obtained in \cite{CaLuPaSi14} and lately justified in \cite{AmGlaTre, AmStTr16}. 
}
For matching-type problems, renormalization heuristics and elliptic PDE descriptions
have led to precise asymptotics and robust structural results in low dimensions,
notably through the PDE/OT approach and subsequent work \cite{AmStTr16,GoldTrev24ran}.
In particular, \cite[Proposition~5.3]{GoldTrev24ran} gives concentration for
$p<d$ under general structural assumptions, including the bipartite matching
problem considered here.  Our matching result extends the conclusion to
$d\le p<d^2/2$ by supplying stronger control of the maximal optimal edge.

From the algorithmic perspective, Euclidean TSP and matching have long served as
benchmarks for local-improvement heuristics (e.g.\ $2$-opt and $k$-opt moves) and
average-case analysis in computer science and operations research, alongside
worst-case PTAS results for geometric optimization \cite{Croes58,LinKernighan73,Arora98,Mitchell99}.

\medskip
The present note contributes to the concentration theory of these models.
While the use of a Poincar\'e (spectral gap) inequality is not new in this setting
(see e.g.\ \cite{GoldTrev24ran}), our contribution is to isolate a quantitative and purely
geometric mechanism converting local swap optimality (cyclical monotonicity for matching,
admissible $2$-edge swaps for TSP) and mesoscopic density into a quantitative
$L^\infty$ bound on edge lengths. This uniform edge control strengthens existing
approaches and provides a direct route to subleading fluctuation bounds
without relying on integrable representations of the optimiser.

\subsection{Extensions and open problems}\label{subsec:extensions}

The mechanism developed in this note is not specific to matching or TSP,
but relies only on two structural ingredients:

\begin{itemize}
\item[(i)] a local swap (or $2$--edge) optimality condition for the minimiser;
\item[(ii)] mesoscopic density control of the underlying random point cloud.
\end{itemize}

We expect the same strategy to apply to a broader class of Euclidean
combinatorial problems whose minimisers satisfy a local exchange property,
provided the corresponding local exchange estimate and annealed bounds are
available.
Examples include minimal spanning trees, $k$--MST, degree-constrained spanning
subgraphs such as $k$--factors, and related matching-type structures.
For these examples it remains necessary to establish the appropriate exchange
inequality and to control its multiplicities; we do not claim those extensions
here {and leave them for future work.}

Although we restrict attention to the cube with uniform sampling for simplicity,
we expect the argument to extend to substantially more general settings.
If the underlying measure has a density $\rho$ that is bounded above and below
on a bounded domain with Lipschitz boundary, the mesoscopic density estimates
should remain valid.
More generally, on compact Riemannian manifolds with bounded geometry,
an analogous argument should be possible provided geodesics satisfy the needed
triangle estimates, small geodesic balls have uniform $r^d$ volume bounds, and
the reference measure has a Poincar\'e inequality.  {Again,} these extensions are not
proved in this paper {but left to future work}.

For simplicity, we also formulated concentration using the $L^2$ Poincar\'e inequality,
bounding the variance by the squared $L^2$ norm of the gradient.
This yields polynomial tail bounds via Chebyshev or Markov inequalities.
Stronger tail estimates (e.g.\ higher polynomial moments or sub-Gaussian behavior
in suitable regimes) could be obtained by combining the same geometric input
with $L^q$ versions of the Poincar\'e inequality or related functional inequalities.
We do not pursue these refinements here, as our focus is on the geometric mechanism.

The main limitation of the present method is the restriction $p<d^2/2$,
which arises from the current power-law control of the maximal edge length.
The exploratory numerical evidence does not display an obvious qualitative
change at $p=d^2/2$ over the tested sizes, but it does not establish behavior
outside the proved range. Establishing concentration for arbitrary $p$ remains, in our view, the central open
problem in this direction.

\medskip
\noindent
\textbf{Organization.}
Section~\ref{sec:matching} proves concentration for bipartite matching. Section~\ref{sec:tsp} treats the monopartite TSP and then the bipartite TSP. Appendix~\ref{sec:numerics} presents exploratory numerical simulations related to Conjecture~\ref{conj:ptoq} and behavior near $p=d^2/2$.

\medskip
\noindent
\textbf{Acknowledgements.}  M.~D'A.~acknowledges support from the ANR project LOUCCOUM (ANR-24-CE40-7809).  M.~D'A. and F.\,M. are grateful to the Department of Mathematics, University of Pisa for excellent working conditions on the occasion of the invitation that started this project (November 2021). F.\,M. acknowledges the support of the ERC Advanced Grant NEITALG, grant agreement No. 101198055. D.T.\ acknowledges the MUR Excellence Department Project awarded to the Department of Mathematics, University of Pisa, CUP I57G22000700001, the HPC Italian National Centre for HPC, Big Data and Quantum Computing - Proposal code CN1 CN00000013, CUP I53C22000690001, the PRIN 2022 Italian grant 2022WHZ5XH - ``understanding the LEarning process of QUantum Neural networks (LeQun)'', CUP J53D23003890006, the project  G24-202 ``Variational methods for geometric and optimal matching problems'' funded by Università Italo Francese.  Research also partly funded by PNRR - M4C2 - Investimento 1.3, Partenariato Esteso PE00000013 - "FAIR - Future Artificial Intelligence Research" - Spoke 1 "Human-centered AI", funded by the European Commission under the NextGeneration EU programme. This research benefitted from the support of the FMJH Program Gaspard Monge for optimization and operations research and their interactions with data science.

\section{Concentration for bipartite matching}
\label{sec:matching}

Let $\mathbf x=(x_i)_{i=1}^n$ and $\mathbf y=(y_j)_{j=1}^n$ be points in $\cube$, and let $p\ge 1$.
The (Euclidean) bipartite matching cost is
\begin{equation}\label{eq:opt_match_prob}
\cC^p_{bM}(\mathbf x,\mathbf y)\coloneqq\min_{\sigma\in\mathcal S_n}\sum_{i=1}^n |x_i-y_{\sigma(i)}|^p.
\end{equation}

We use the following notational conventions throughout.
For a point set $\mathbf z=(z_i)_{i=1}^n$ and a measurable set $\Omega\subseteq\cube$ we write
\[
N^{\mathbf z}_\Omega\coloneqq\#\{\,i\in\{1,\dots,n\}: z_i\in\Omega\,\}.
\]
We denote by $B(x,r)$ the Euclidean ball in $\R^d$ of centre $x$ and radius $r$;
occupation counts below explicitly use $B(x,r)\cap\cube$.

In the random case $\mathbf x = \mathbf X$, $\mathbf y = \mathbf Y$, we assume that $x_i = X_i$ and $y_j = Y_j$ are realizations of i.i.d.\ random variables uniformly distributed on $\cube$. The matching cost becomes a random variable $\cC^p_{bM}(\mathbf X, \mathbf Y)$, and its annealed asymptotics are by now standard \cite{AKT84,Ta92,AmStTr16,dBGM99,FoGu15}: as $n \to \infty$,
\begin{equation}\label{eq:match_bound}
\EE[\cC^p_{bM}(\mathbf X,\mathbf Y)] \asymp
\footnote{
{
We write $A\lesssim B$ if there exists a constant $C>0$, that may only depend on $d$ and $p$, such that $A\le CB$. Similarly, we write $A \gtrsim B$ is there exists a constant $C>0$, that may only depend on $d$ and $p$, such that $A\ge CB$. We write $A\asymp B$ if both
$A\lesssim B$ and $A \gtrsim B$.
}
}
n
\begin{cases}
n^{ - \frac p 2} & \text{if $d = 1$}, \\
\big( \frac{\log n}n \big)^{\frac p 2} & \text{if $d =2$}, \\
n^{ - \frac p d} & \text{if $d \ge 3$},
\end{cases}
\end{equation}
where the constants in $\asymp$ may depend on $d$ and $p$; existence of a
limit after rescaling remains open only in $d=2$, $p\neq2$.  The exceptional
scaling in dimensions $d=1$ and $d=2$ is due to anomalously long optimal
matching edges, {and are not dealt with in the present paper. We focus here on $d\ge3$}.

\begin{theorem}[\textsc{Concentration for bipartite matching}]\label{thm:conc_matching}
Let $d\ge 3$ and $1\le p<d^2/2$. Then, there exist constants $\theta = \theta(p,d)>0$ and $C=C(p,d)<\infty$ such that for all $n\ge 1$ and $\lambda>0$,
\[
\PP\!\left(\left| \cC^p_{bM} (\mathbf{X}, \mathbf{Y})  - \EE[ \cC^p_{bM} (\mathbf{X}, \mathbf{Y}) ] \right| \ge  \lambda n^{1-\frac p d}\right)
 \le C  n^{- \theta} \lambda^{-2} .
 \]
\end{theorem}

\subsection{2-opt inequality and local edge-to energy control}

Before we provide the proof of Theorem~\ref{thm:conc_matching}, we focus on the key steps, described informally in the previous section, that allow us to obtain a uniform control on edge costs, starting from a $2$--opt inequality. In the case of bipartite matching, the optimality condition we use is the following two-point swap inequality, which is a special instance of cyclical monotonicity in optimal transport theory, 
{
see for instance \cite[Theorem 368]{HardyLittlewoodPolya1952}.
}

\begin{lemma}[$2$--opt inequality]\label{lem:2opt_matching}
Let $\sigma$ be an optimiser in \eqref{eq:opt_match_prob}. Then for all $i,j\in\{1,\dots,n\}$,
\begin{equation}\label{eq:2opt}
|x_i-y_{\sigma(i)}|^p + |x_j-y_{\sigma(j)}|^p \le |x_i-y_{\sigma(j)}|^p + |x_j-y_{\sigma(i)}|^p.
\end{equation}
\end{lemma}

\begin{proof}
Otherwise swapping $\sigma(i)$ and $\sigma(j)$ strictly decreases the total cost.
\end{proof}

We now make the $2$--opt condition quantitative by turning it into a uniform bound of a single edge cost by local energy.
In the semi-discrete matching (where one cloud of points is replaced with the uniform density), this type of local-energy control goes back at least to \cite[Lemma 4.4]{AmGlaTre} and \cite[Lemma 4.1]{GO}; the lemma below adapts this viewpoint to the present two-cloud setting.

\begin{lemma}\label{lem:l_infty_to_loc}
For any $p>1$, there exists $\epsilon=\epsilon(p)>0$ and $C=C(p)<\infty$ such that the following holds. Let $\mathbf{x} = (x_i)_{i=1}^n, \mathbf{y} = (y_j)_{j=1}^n \subseteq \cube$ be two sets of points, and let $\sigma$ be an optimiser for the $p$-cost. For any $i\in\{1,\dots,n\}$, set
\[
B_i = B\!\left( \frac{x_i+y_{\sigma(i)} }{2} , \epsilon\, |x_i- y_{\sigma(i)}| \right).
\]
Then,
\begin{equation}\label{eq:L-infty-loc-energy}
N^{\mathbf{x}}_{B_i} \cdot |x_i- y_{\sigma(i)}|^{p}
\;\le\; C \sum_{x_j \in B_i} |x_j - y_{\sigma (j)}|^p.
\end{equation}
\end{lemma}

\begin{proof}[Proof of Lemma \ref{lem:l_infty_to_loc}]
Write $z_i=(x_i+y_{\sigma(i)})/2$ and assume that $0<\epsilon<1/2$. 
{
Note that if $B_i\cap\mathbf x=\varnothing$ then \eqref{eq:L-infty-loc-energy} is trivial. Let $x_j\in B_i$. Suppose that $B_i\cap\mathbf x \neq \varnothing$ and let $x_j \in B_i$.
}
It suffices to show that \eqref{eq:2opt} and the triangle inequality imply
\begin{equation}\label{eq:Linfty-inf-ineq}
|x_i-y_{\sigma(i)}|^p \le C \,|x_j-y_{\sigma(j)}|^p.
\end{equation}
Indeed, summing \eqref{eq:Linfty-inf-ineq} over $x_j\in B_i$ yields \eqref{eq:L-infty-loc-energy}. By the triangle inequality,
\begin{align}
|x_i-y_{\sigma (j)}|
&\le |x_i-z_i| + |z_i-x_j| + |x_j-y_{\sigma(j)}|
\le (1 + \epsilon) |x_i-z_i| + |x_j-y_{\sigma(j)}|,\label{eq:triang-1}\\
|x_j-y_{\sigma(i)}|
&\le |x_j-z_i| + |z_i-y_{\sigma (i)}|
\le (1 + \epsilon) |x_i-z_i|,\label{eq:triang-2}
\end{align}
{
see figure \ref{fig:geo-pic} for a geometrical picture.
}

\medskip
Using the inequality $(a+b)^p \le (1+\eta)a^p + C_{\eta,p}\,b^p$ with $\eta>0$ small, together with \eqref{eq:2opt}, \eqref{eq:triang-1} and \eqref{eq:triang-2}, we obtain
\[
|x_i-y_{\sigma(i)}|^p + |x_j-y_{\sigma(j)}|^p
\le (2+\eta)(1+\epsilon)^p |x_i-z_i|^p + C |x_j-y_{\sigma(j)}|^p.
\]
Since $|x_i-z_i|=\frac12|x_i-y_{\sigma(i)}|$, choosing $\epsilon=\epsilon(p)$ (and then $\eta=\eta(p)$) so that
$1-(2+\eta)(1+\epsilon)^p2^{-p}>0$, we can rearrange {the above display} to get \eqref{eq:Linfty-inf-ineq}.
\begin{figure}[ht!]
  \centering
  \begin{tikzpicture}[scale=0.95]
    \tikzstyle{xpoint}=[circle, fill=blue!70, inner sep=1.6pt]
    \tikzstyle{ypoint}=[circle, fill=orange!85!black, inner sep=1.6pt]
    \tikzstyle{midpoint}=[circle, fill=black, inner sep=1.4pt]

    \node[xpoint, label=below:{$x_i$}] (xi) at (0.2,0.1) {};
    \node[ypoint, label=right:{$y_{\sigma(i)}$}] (yi) at (8.0,2.0) {};
    \draw[thick] (xi) -- (yi);

    \coordinate (z) at ($(xi)!0.5!(yi)$);
    \node[midpoint, label=above:{$z_i$}] at (z) {};

    \def\radius{2.5}
    \draw[dashed, thick] (z) circle [radius=\radius];
    \draw[<->, red!80!black, thick] (z) -- ++(-35:\radius)
      node[midway, below, sloped, text=black] {$\epsilon\,|x_i-y_{\sigma(i)}|$};

    \node[xpoint, label=above:{$x_j$}] (xj) at (4.6,1.55) {};
    \node[ypoint, label=above:{$y_{\sigma(j)}$}] (yj) at (2.0,3.35) {};
    \draw[thick] (xj) -- (yj);
    \draw[gray!70, densely dotted] (z) -- (xj);
  \end{tikzpicture}
  \caption{Local geometry used in Lemma \ref{lem:l_infty_to_loc}.}
  \label{fig:geo-pic}
\end{figure}
\end{proof}

\begin{lemma}[Poincar\'e inequality on the cube]\label{lem:cube_poincare}
Let $\mu_m$ be {the} uniform probability measure on $(\cube)^m$.  Every locally
Lipschitz $F:(\cube)^m\to\mathbb R$ satisfies
\[
\operatorname{Var}_{\mu_m}(F)
\le C_{\rm P}\int_{(\cube)^m}|\nabla F|^2\,d\mu_m,
\]
where $C_{\rm P}$ is independent of $m$.
\end{lemma}

\begin{proof}
The one-dimensional Wirtinger inequality for uniform measure on $[0,1]$
has constant $1/\pi^2$, 
{
see for instance \cite[Proposition 4.5.5]{BGL14}.  
}
Applying it conditionally in each scalar coordinate
and tensorizing gives the assertion, with $C_{\rm P}=1/\pi^2$; see
\cite[Section~3.1]{BGL14}.  Approximation extends the inequality from smooth
functions to locally Lipschitz functions.
\end{proof}

\begin{lemma}[Uniform mesoscopic occupancy]\label{lem:occupancy}
Fix $d\ge1$, $\alpha\in(0,1)$ and $r_*\in(0,1/8)$.  There is $c_d>0$ such
that, for i.i.d.\ uniform points $\mathbf X=(X_i)_{i=1}^n$ and
\[
A_n=\left\{N^{\mathbf X}_{B(x,r)\cap\cube}\ge c_dnr^d
 :x\in\cube,\ n^{-\alpha/d}\le r\le r_*\right\},
\]
one has, for every $\beta>0$,
\begin{equation}\label{eq:quant-diffusion-unif}
\PP(A_n^c)\le C_{d,\alpha,\beta,r_*}n^{-\beta}.
\end{equation}
\end{lemma}

\begin{proof}
Put $\rho=n^{-\alpha/d}$ and use dyadic radii
$r_k=2^k\rho$ up to $2r_*$.  A geometric constant $a_d\in(0,1/16)$ has the
following property: for every $x\in\cube$ and $r\le2r_*$, the set
$B(x,r)\cap\cube$ contains a ball of radius $4a_dr$ whose centre is in
$\cube$.  Approximate that centre by a point of an $a_dr$--net of the cube.
Consequently $B(x,r)\cap\cube$ contains one of the net balls of radius
$2a_dr$.

There are at most $C_dr_k^{-d}$ net balls at level $k$, and each has volume
at least $c_dr_k^d$, including near the boundary.  A multiplicative Chernoff
bound therefore shows that the probability that any level-$k$ net ball
contains fewer than half its mean number of sample points is at most
\[
C_dr_k^{-d}\exp(-c_dnr_k^d)
\le C_dn^\alpha\exp(-c_dn^{1-\alpha}).
\]
There are $O(\log n)$ levels.  Their union has probability smaller than
$C_\beta n^{-\beta}$ for every $\beta$.  For an arbitrary
$r\in[\rho,r_*]$, choose a dyadic level $r_k\in[r/2,r]$ and apply the
construction to $B(x,r_k)\cap\cube\subseteq B(x,r)\cap\cube$.  After
decreasing $c_d$, the contained net ball gives
$N_{B(x,r)\cap\cube}^{\mathbf X}\ge c_dnr^d$.
\end{proof}

\subsection{Proof of Theorem \texorpdfstring{\ref{thm:conc_matching}}{\ref*{thm:conc_matching}}}

\begin{proof}
Write $C_p=\cC^p_{bM}(\mathbf X,\mathbf Y)$.  The minimum of the finitely
many locally Lipschitz matching costs is locally Lipschitz.  Thus it is
differentiable at Lebesgue-a.e.\ configuration by Rademacher's theorem.  At
such a configuration, for an optimal permutation $\sigma$,
\[
\nabla_{x_i}C_p=p|x_i-y_{\sigma(i)}|^{p-2}(x_i-y_{\sigma(i)}),
\quad
\nabla_{y_{\sigma(i)}}C_p=-\nabla_{x_i}C_p,
\]
with the usual a.e.\ interpretation when $p=1$.  Hence, with
$\ell_i=|x_i-y_{\sigma(i)}|$,
\begin{equation}\label{eq:matching-gradient}
|\nabla C_p|^2\le 2p^2\sum_i\ell_i^{2p-2}.
\end{equation}

Suppose first that $1\le p\le2$ and put $q=2p-2$.  For $p>1$, H\"older's
inequality and the annealed upper bound in \eqref{eq:match_bound} give
\[
\EE\sum_i\ell_i^q
\le n^{1-q/p}\EE[C_p^{q/p}]
\le n^{1-q/p}(\EE C_p)^{q/p}
\lesssim n^{1-q/d}.
\]
For $p=1$ the left-hand side of \eqref{eq:matching-gradient} is bounded by
$2n$, which is the same estimate with $q=0$.  Lemma~\ref{lem:cube_poincare}
and Chebyshev's inequality now yield
\[
\PP\bigl(|C_p-\EE C_p|\ge\lambda n^{1-p/d}\bigr)
\lesssim\lambda^{-2}n^{-1+2/d}.
\]
This proves the result in this range because $d\ge3$.

Now let $p>2$, and set $L=\max_i\ell_i$.  From
\eqref{eq:matching-gradient},
\begin{equation}\label{eq:matching-gradient-max}
|\nabla C_p|^2\lesssim C_pL^{p-2}.
\end{equation}
The exponents in the next Cauchy--Schwarz step are important.  If $\sigma_{2p}$
is optimal for exponent $2p$, optimality at exponent $p$ followed by
Cauchy--Schwarz gives the deterministic inequality
\[
C_p^2\le\left(\sum_i|x_i-y_{\sigma_{2p}(i)}|^p\right)^2
\le n\cC^{2p}_{bM}(\mathbf x,\mathbf y).
\]
Consequently,
\begin{align}
\EE[C_pL^{p-2}]
&\le(\EE C_p^2)^{1/2}(\EE L^{2(p-2)})^{1/2}\notag\\
&\le n^{1/2}(\EE\cC^{2p}_{bM})^{1/2}
       (\EE L^{2(p-2)})^{1/2}\notag\\
&\lesssim n^{1-p/d}(\EE L^{2(p-2)})^{1/2}.
\label{eq:matching-cs}
\end{align}
The last line uses \eqref{eq:match_bound} with exponent $2p$.
Lemma~\ref{lem:cube_poincare}, \eqref{eq:matching-gradient-max}, and
Chebyshev therefore imply
\begin{equation}\label{eq:proof_conc_step_1}
\PP\bigl(|C_p-\EE C_p|\ge\lambda n^{1-p/d}\bigr)
\lesssim\lambda^{-2}n^{p/d-1}(\EE L^{2(p-2)})^{1/2}.
\end{equation}

It remains to bound $L$.  Fix $\alpha<p/(p+d)$ and define
$\alpha'=\alpha(p+d)/p<1$.  Let $A_n$ be the event in
Lemma~\ref{lem:occupancy}, with a fixed sufficiently small $r_*=r_*(d)$, and
let
\[
B_n=\{C_p\le n^{1-\alpha'p/d}\}.
\]
For every integer $s>1$, choosing an optimiser for exponent $ps$ and applying
H\"older gives
\[
C_p^s\le n^{s-1}\cC^{ps}_{bM}.
\]
Markov's inequality and \eqref{eq:match_bound} consequently give
\begin{equation}\label{eq:int-step-matchcost}
\PP(B_n^c)\lesssim n^{(\alpha'-1)ps/d}.
\end{equation}
Thus both $A_n^c$ and $B_n^c$ have arbitrarily high polynomial decay after
choosing $s$ large.

On $A_n\cap B_n$, fix an optimal edge of length $\ell$ and put
$r=\epsilon\ell$, with $\epsilon$ from Lemma~\ref{lem:l_infty_to_loc}.  If
$r\le n^{-\alpha/d}$ there is nothing to prove.  If
$n^{-\alpha/d}<r\le r_*$, Lemmas~\ref{lem:occupancy} and
\ref{lem:l_infty_to_loc} give
\[
n\ell^{p+d}\lesssim C_p\le n^{1-\alpha'p/d},
\]
so $\ell\lesssim n^{-\alpha/d}$ because
$\alpha(p+d)=\alpha'p$.  Finally, if $r>r_*$, the radius-$r_*$ ball with the
same centre lies inside $B_i$.  The same two lemmas gives
$nr_*^d\ell^p\lesssim C_p$.  Since $\ell>r_*/\epsilon$, this contradicts
$C_p\le n^{1-\alpha'p/d}$ for all sufficiently large $n$.  Enlarging the
constant handles the remaining $n$.  We have proved
\begin{equation}\label{eq:sup_est_conc}
L\lesssim_\alpha n^{-\alpha/d}
\quad\text{on }A_n\cap B_n.
\end{equation}

The cube has bounded diameter, and the exceptional events have arbitrarily
high polynomial decay; hence
$(\EE L^{2(p-2)})^{1/2}\lesssim n^{-\alpha(p-2)/d}$.  Substitution into
\eqref{eq:proof_conc_step_1} gives
\begin{equation}\label{eq:fin_est_conc}
\PP\bigl(|C_p-\EE C_p|\ge\lambda n^{1-p/d}\bigr)
\lesssim\lambda^{-2}n^{p/d-1-\alpha(p-2)/d}.
\end{equation}
An $\alpha<p/(p+d)$ makes the exponent negative precisely when
\[
p\left(1-\frac{p-2}{p+d}\right)<d
\quad\Longleftrightarrow\quad 2p<d^2.
\]
Choosing such an $\alpha$ proves the theorem.
\end{proof}

\section{Concentration for the traveling salesperson problem}
\label{sec:tsp}
We now give complete versions of the argument for the mono- and bipartite
traveling salesperson problems.

\subsection{Monopartite case}

For $n\ge3$, a Hamiltonian cycle on $\{1,\ldots,n\}$ is represented by a
cyclic ordering $\tau$.  For points $\mathbf x=(x_i)_{i=1}^n\subset\cube$ set
\[
L_p(\tau;\mathbf x)=\sum_{k=1}^n
 |x_{\tau(k)}-x_{\tau(k+1)}|^p,
\qquad \tau(n+1)=\tau(1),
\]
and
\begin{equation}\label{eq:TSP_def}
\cC^p_{TS}(\mathbf x)=\min_\tau L_p(\tau;\mathbf x).
\end{equation}
For every fixed $d\ge2$ and $q\ge1$, standard Euclidean-functional bounds
give
\begin{equation}\label{eq:tsp_annealed}
\EE[\cC^q_{TS}(\mathbf X)]\lesssim_{d,q}n^{1-q/d};
\end{equation}
see \cite{BHH59,Steele97,Yukich98}.  Only this upper bound, rather than the
existence of a rescaled limit, is used below.

\begin{theorem}[\textsc{Concentration for Euclidean TSP}]\label{thm:conc_tsp}
Let $d\ge3$ and $1\le p<d^2/2$.  There are
$\theta=\theta(p,d)>0$ and $C=C(p,d)<\infty$ such that, for every $n\ge4$
and $\lambda>0$,
\[
\PP\!\left(\left|\cC^p_{TS}(\mathbf X)-\EE[\cC^p_{TS}(\mathbf X)]\right|
 \ge\lambda n^{1-p/d}\right)
\le Cn^{-\theta}\lambda^{-2}.
\]
\end{theorem}

A $2$--opt move removes two tour edges and reconnects the resulting paths to
obtain another Hamiltonian cycle.

\begin{lemma}[$2$--opt inequality]\label{lem:2opt_tsp}
Let $\tau$ be an optimiser in \eqref{eq:TSP_def}.  For two distinct tour
edges, written in cyclic order as
$(x_{\tau(a)},x_{\tau(a+1)})$ and
$(x_{\tau(b)},x_{\tau(b+1)})$, one has
\begin{equation}\label{eq:2opt_tsp}
|x_{\tau(a)}-x_{\tau(a+1)}|^p+|x_{\tau(b)}-x_{\tau(b+1)}|^p
\le |x_{\tau(a)}-x_{\tau(b)}|^p
   +|x_{\tau(a+1)}-x_{\tau(b+1)}|^p.
\end{equation}
\end{lemma}

\begin{proof}
If the two edges are disjoint, deleting them leaves two paths.  Adding
$(x_{\tau(a)},x_{\tau(b)})$ and
$(x_{\tau(a+1)},x_{\tau(b+1)})$ joins the paths into one Hamiltonian cycle;
optimality gives \eqref{eq:2opt_tsp}.  If the edges are adjacent, the displayed
reconnection merely restores the two edges with reversed orientation, so the
inequality is equality.
\end{proof}

\begin{lemma}[Local edge-to-energy control for TSP]\label{lem:tsp_edge_to_loc}
Let $p>1$ and $n\ge4$.  There are $\epsilon=\epsilon(p)\in(0,1/2)$ and
$C=C(p)$ with the following property.  For an optimal tour $\tau$, fix
\[
e=(x_{\tau(a)},x_{\tau(a+1)}),\quad
\ell=|x_{\tau(a+1)}-x_{\tau(a)}|,
\quad B_e=B\!\left(\tfrac{x_{\tau(a)}+x_{\tau(a+1)}}2,\epsilon\ell\right).
\]
Let $\mathcal E_\tau(B_e)$ be the set of tour edges having at least one
endpoint in $B_e$.  Then
\begin{equation}\label{eq:tsp_edge_to_loc_energy}
N^{\mathbf x}_{B_e\cap\cube}\ell^p
\le C\sum_{f\in\mathcal E_\tau(B_e)}|f|^p.
\end{equation}
\end{lemma}

\begin{proof}
The assertion is trivial if $\ell=0$.  Fix a vertex
$u=x_{\tau(b)}\in B_e$.  Since $\epsilon<1/2$, $u$ is neither endpoint of
$e$.  Of the two edges incident to $u$, at least one, say $f=(u,v)$ after
choosing its orientation, is disjoint from $e$.  Apply
Lemma~\ref{lem:2opt_tsp} to $e$ and $f$, orienting the cyclic order so that
the inserted edges join $x_{\tau(a)}$ to $u$ and $x_{\tau(a+1)}$ to $v$.
Optimality gives
\[
\ell^p+|u-v|^p
\le |x_{\tau(a)}-u|^p+|x_{\tau(a+1)}-v|^p.
\]
Writing $z_e$ for the midpoint of $e$, the triangle inequality gives
\[
|x_{\tau(a)}-u|\le(\tfrac12+\epsilon)\ell,\qquad
|x_{\tau(a+1)}-v|\le(\tfrac12+\epsilon)\ell+|u-v|.
\]
For any $\eta>0$, $(s+t)^p\le(1+\eta)s^p+C_{\eta,p}t^p$.
Consequently
\[
\ell^p+|u-v|^p
\le(2+\eta)(\tfrac12+\epsilon)^p\ell^p+C_{\eta,p}|u-v|^p.
\]
Because $p>1$, first choose $\epsilon$ and then $\eta$ so that
$(2+\eta)(1/2+\epsilon)^p<1$.  Absorption yields
$\ell^p\le C|f|^p$.  Summing over $u\in B_e$ proves
\eqref{eq:tsp_edge_to_loc_energy}: each tour edge can be selected by at most
its two endpoints, so its multiplicity is at most two.
\end{proof}

\begin{proof}[Proof of Theorem~\ref{thm:conc_tsp}]
Write $T_p=\cC^p_{TS}(\mathbf X)$ and let $\tau$ be an optimal tour.  The
functional is locally Lipschitz and hence differentiable at a.e.\ configuration.
At such a point, the derivative at a vertex is the sum of the derivatives of
its two incident edges.  The inequality $|a+b|^2\le2|a|^2+2|b|^2$ and the
fact that every edge has two endpoints give
\begin{equation}\label{eq:tsp-gradient}
|\nabla T_p|^2\lesssim_p\sum_{e\in\tau}|e|^{2p-2}.
\end{equation}
For $1\le p\le2$, put $q=2p-2$.  H\"older and
\eqref{eq:tsp_annealed} yield
\[
\EE\left[\sum_{e\in\tau}|e|^q\right]
\le n^{1-q/p}(\EE T_p)^{q/p}\lesssim n^{1-q/d},
\]
with the direct bound $n$ when $p=1$.  Poincar\'e and Chebyshev then give
$\PP(|T_p-\EE[T_p]|\ge\lambda n^{1-p/d})\lesssim
\lambda^{-2}n^{-1+2/d}$.

Assume now $p>2$ and let $L=\max_{e\in\tau}|e|$.  From
\eqref{eq:tsp-gradient}, $|\nabla T_p|^2\lesssim T_pL^{p-2}$.  Choosing an
optimal tour for exponent $2p$ gives $T_p^2\le n\cC^{2p}_{TS}$, and hence
\begin{equation}\label{eq:tsp-poincare-step}
\PP(|T_p-\EE[T_p]|\ge\lambda n^{1-p/d})
\lesssim\lambda^{-2}n^{p/d-1}(\EE[L^{2(p-2)}])^{1/2}.
\end{equation}
Here we used \eqref{eq:tsp_annealed}, Lemma~\ref{lem:cube_poincare}, and
Cauchy--Schwarz exactly as in \eqref{eq:matching-cs}.

Fix $\alpha<p/(p+d)$, put $\alpha'=\alpha(p+d)/p$, and let $A_n$ be the
event of Lemma~\ref{lem:occupancy}.  Set
$B_n^{TS}=\{T_p\le n^{1-\alpha'p/d}\}$.  For each integer $s>1$,
$T_p^s\le n^{s-1}\cC^{ps}_{TS}$; therefore \eqref{eq:tsp_annealed} and
Markov give $\PP((B_n^{TS})^c)\lesssim n^{(\alpha'-1)ps/d}$.
On $A_n\cap B_n^{TS}$, Lemma~\ref{lem:tsp_edge_to_loc} and the bounded
multiplicity in its right-hand side imply, for every edge of length $\ell$
with $n^{-\alpha/d}<\epsilon\ell\le r_*$,
\[
n \ell^{p+d}\lesssim\sum_{f\in\mathcal E_\tau(B_e)}|f|^p
\le2T_p\le2n^{1-\alpha'p/d}.
\]
If $\epsilon \ell>r_*$, applying occupancy at $r_*$ instead gives
$nr_*^d\ell^p\lesssim T_p$, contradicting the last bound for large $n$.
Thus $L\lesssim n^{-\alpha/d}$ off an event of arbitrarily high polynomial
decay.  Since all edges are bounded by $\sqrt d$, we obtain
$(\EE [L^{2(p-2)}])^{1/2}\lesssim n^{-\alpha(p-2)/d}$.
Substitution into \eqref{eq:tsp-poincare-step} yields the exponent
$p/d-1-\alpha(p-2)/d$, which can be made negative exactly when
$2p<d^2$.  This proves the theorem.
\end{proof}

\subsection{Bipartite case}

Let $\mathbf x=(x_i)_{i=1}^n$ and $\mathbf y=(y_j)_{j=1}^n$.  A bipartite
Hamiltonian cycle is a Hamiltonian cycle in the complete bipartite graph with
vertex classes $\{x_1,\ldots,x_n\}$ and $\{y_1,\ldots,y_n\}$.  Denote the
set of these alternating cycles by $\mathcal H_{\rm alt}(\mathbf x,\mathbf y)$
and define
\begin{equation}\label{eq:bTSP_def}
\cC^p_{bTS}(\mathbf x,\mathbf y)
=\min_{\gamma\in\mathcal H_{\rm alt}(\mathbf x,\mathbf y)}
  \sum_{e\in\gamma}|e|^p.
\end{equation}
Equivalently, a cycle has the form
$x_{\pi(1)}\to y_{\sigma(1)}\to x_{\pi(2)}\to\cdots\to
y_{\sigma(n)}\to x_{\pi(1)}$ for two permutations $\pi,\sigma$, modulo
cyclic symmetry.  In particular, the order of the $x$--vertices is not fixed.

For later use we record the annealed upper bound
\begin{equation}\label{eq:btsp-annealed}
\EE[\cC^q_{bTS}(\mathbf X,\mathbf Y)]\lesssim_{d,q}n^{1-q/d},
\qquad d\ge3, q\ge1.
\end{equation}
Indeed, take an optimal $q$--matching $x_i\leftrightarrow y_{\sigma(i)}$
and an optimal $q$--TSP order $\pi$ of the $x$--points.  The alternating
cycle $x_{\pi(k)}\to y_{\sigma(\pi(k))}\to x_{\pi(k+1)}$ has, by
$(a+b)^q\le2^{q-1}(a^q+b^q)$, cost at most a constant times the sum of the
matching and monopartite-TSP $q$--costs.  Taking expectations and using
\eqref{eq:match_bound} and \eqref{eq:tsp_annealed} proves
\eqref{eq:btsp-annealed}.  This verifies the scaling input for precisely the
standard alternating model in \eqref{eq:bTSP_def}; the low-dimensional exact
computations in \cite{caracciolo2018solution,capelli2018exact} are not used.

\begin{theorem}[\textsc{Concentration for bipartite TSP}]\label{thm:conc_bTSP}
Let $d\ge3$ and $1\le p<d^2/2$.  There are
$\theta=\theta(p,d)>0$ and $C=C(p,d)<\infty$ such that, for all $n\ge3$ and
$\lambda>0$,
\begin{equation}\label{eq:concen_bTSP}
\PP\!\left(\left|\cC^p_{bTS}(\mathbf X,\mathbf Y)
-\EE[\cC^p_{bTS}(\mathbf X,\mathbf Y)]\right|
\ge\lambda n^{1-p/d}\right)
\le Cn^{-\theta}\lambda^{-2}.
\end{equation}
\end{theorem}

\begin{lemma}[Admissible alternating reconnection]\label{lem:btsp-reconnect}
Let $\gamma$ be an alternating Hamiltonian cycle and let $e=(x_i,y_j)$ be
one of its edges.  For any $x_k\ne x_i$, consider the two edges incident to
$x_k$.  Among the two alternating reconnections involving $e$ and one of
those incident edges, one produces a single alternating Hamiltonian cycle.
\end{lemma}

\begin{proof}
Orient $\gamma$ so that $e$ is traversed from $x_i$ to $y_j$.  Exactly one
edge incident to $x_k$ is traversed into $x_k$; write it $(y_a,x_k)$.  Delete
$(x_i,y_j)$ and $(y_a,x_k)$.  If $y_a=y_j$, these edges are adjacent and the
reconnection simply restores the original cycle.  Otherwise, what remains consists of the two paths from
$y_j$ to $y_a$ and from $x_k$ to $x_i$.  Adding $(y_j,x_k)$ and $(y_a,x_i)$
joins the two paths into one cycle.  Both new edges join unlike colour
classes, every vertex again has degree two, and all vertices lie in the
joined component.  The result is therefore an alternating Hamiltonian cycle.
The other incident edge may close the two paths separately, which explains
why the choice is needed.
\end{proof}

\begin{figure}[ht]
\centering
\begin{subfigure}{.46\textwidth}\centering
\begin{tikzpicture}[x=1cm,y=.8cm]
\tikzstyle{xpt}=[circle,fill=blue!65,inner sep=1.5pt]
\tikzstyle{ypt}=[circle,fill=orange!85!black,inner sep=1.5pt]
\node[xpt,label=below:$x_i$] (xi) at (0,0) {};
\node[ypt,label=above:$y_j$] (yj) at (1,1) {};
\node[ypt,label=above:$y_a$] (ya) at (3,1) {};
\node[xpt,label=below:$x_k$] (xk) at (4,0) {};
\draw[very thick,dashed] (xi)--(yj); \draw[very thick,dashed] (ya)--(xk);
\draw[thick] (yj) to[bend left=20] (ya); \draw[thick] (xk) to[bend left=25] (xi);
\end{tikzpicture}
\caption{Cut the two oriented edges.}
\end{subfigure}\hfill
\begin{subfigure}{.46\textwidth}\centering
\begin{tikzpicture}[x=1cm,y=.8cm]
\tikzstyle{xpt}=[circle,fill=blue!65,inner sep=1.5pt]
\tikzstyle{ypt}=[circle,fill=orange!85!black,inner sep=1.5pt]
\node[xpt,label=below:$x_i$] (xi) at (0,0) {};
\node[ypt,label=above:$y_j$] (yj) at (1,1) {};
\node[ypt,label=above:$y_a$] (ya) at (3,1) {};
\node[xpt,label=below:$x_k$] (xk) at (4,0) {};
\draw[thick] (yj) to[bend left=20] (ya); \draw[thick] (xk) to[bend left=25] (xi);
\draw[very thick,dotted,red] (yj)--(xk); \draw[very thick,dotted,red] (ya)--(xi);
\end{tikzpicture}
\caption{The selected reconnection is one cycle.}
\end{subfigure}
\caption{The admissible alternating reconnection in
Lemma~\ref{lem:btsp-reconnect}.}
\label{fig:btsp-reconnection}
\end{figure}

For an optimal $\gamma$, the admissible reconnection in
Lemma~\ref{lem:btsp-reconnect} satisfies
\begin{equation}\label{eq:btsp-cost-swap}
|x_i-y_j|^p+|x_k-y_a|^p
\le |x_k-y_j|^p+|x_i-y_a|^p,
\end{equation}
because otherwise the reconnected cycle would have smaller cost.

\begin{lemma}[Local edge-to-energy control for bipartite TSP]
\label{lem:btsp_edge_to_loc_2}
For $p>1$ there are $\epsilon=\epsilon(p)\in(0,1/2)$ and $C=C(p)$ such
that the following holds.  Let $\gamma$ be an optimal alternating cycle and
$e=(x_i,y_j)\in\gamma$.  With
\[
\ell=|x_i-y_j|,\qquad z_e=(x_i+y_j)/2,\qquad
B_e=B(z_e,\epsilon\ell),
\]
one has
\begin{equation}\label{eq:btsp-local-energy}
N^{\mathbf x}_{B_e\cap\cube}\ell^p
\le C\sum_{f\in\gamma:\ f\text{ is selected at some }x_k\in B_e}|f|^p
\le C\sum_{f\in\gamma}|f|^p.
\end{equation}
\end{lemma}

\begin{proof}
If $\ell=0$ there is nothing to prove.  Fix $x_k\in B_e$.  Since
$\epsilon<1/2$, it is not the endpoint $x_i$.  Orient $e$ from $x_i$ to
$y_j$ and use Lemma~\ref{lem:btsp-reconnect} to select the incident edge
$(y_a,x_k)$.  If $y_a=y_j$, then
\[
|x_k-y_a|=|x_k-y_j|\ge(\tfrac12-\epsilon)\ell,
\]
so directly $\ell^p\le(1/2-\epsilon)^{-p}|x_k-y_a|^p$.
Suppose henceforth that $y_a\ne y_j$.  Inequality
\eqref{eq:btsp-cost-swap} and the midpoint geometry
give
\[
\ell^p+|x_k-y_a|^p
\le |x_k-y_j|^p+|x_i-y_a|^p,
\]
\[
|x_k-y_j|\le(\tfrac12+\epsilon)\ell,\qquad
|x_i-y_a|\le(\tfrac12+\epsilon)\ell+|x_k-y_a|.
\]
Using $(u+v)^p\le(1+\eta)u^p+C_{\eta,p}v^p$ and choosing
$\epsilon,\eta$ so that $(2+\eta)(1/2+\epsilon)^p<1$, we absorb the long-edge
term and obtain $\ell^p\le C|x_k-y_a|^p$.  Summing over the $x_k$ in the
ball proves the first inequality in \eqref{eq:btsp-local-energy}.  Each
$x$--vertex has a unique incoming edge in the chosen orientation, so a tour
edge is selected at most once; the second inequality follows.
\end{proof}

\begin{proof}[Proof of Theorem~\ref{thm:conc_bTSP}]
Write $S_p=\cC^p_{bTS}(\mathbf X,\mathbf Y)$ and choose an optimal alternating
cycle $\gamma$.  At a.e.\ configuration the locally Lipschitz functional is
differentiable.  Each vertex has two incident edges, so
\begin{equation}\label{eq:btsp-gradient}
|\nabla S_p|^2\lesssim_p\sum_{e\in\gamma}|e|^{2p-2}.
\end{equation}
For $1\le p\le2$, with $q=2p-2$, H\"older and
\eqref{eq:btsp-annealed} give
\[
\EE\left[\sum_{e\in\gamma}|e|^q\right]
\le(2n)^{1-q/p}(\EE [S_p])^{q/p}\lesssim n^{1-q/d},
\]
with the direct bound $2n$ at $p=1$.  Poincar\'e and Chebyshev yield the
required bound with exponent $-1+2/d<0$.

Let $p>2$ and $L=\max_{e\in\gamma}|e|$.  Then
$|\nabla S_p|^2\lesssim S_pL^{p-2}$.  Choosing a $2p$--optimal alternating
cycle gives $S_p^2\le2n\cC^{2p}_{bTS}$.  Thus
\begin{equation}\label{eq:btsp-poincare-step}
\PP(|S_p-\EE [S_p]|\ge\lambda n^{1-p/d})
\lesssim\lambda^{-2}n^{p/d-1}(\EE [L^{2(p-2)}])^{1/2},
\end{equation}
by \eqref{eq:btsp-annealed}, Cauchy--Schwarz, and
Lemma~\ref{lem:cube_poincare}.

Fix $\alpha<p/(p+d)$ and $\alpha'=\alpha(p+d)/p$.  Take the event $A_n$
from Lemma~\ref{lem:occupancy} for the $X$--cloud and set
$B_n^{bTS}=\{S_p\le n^{1-\alpha'p/d}\}$.  For every integer $s>1$,
$S_p^s\le(2n)^{s-1}\cC^{ps}_{bTS}$, so \eqref{eq:btsp-annealed} and Markov
give $\PP((B_n^{bTS})^c)\lesssim n^{(\alpha'-1)ps/d}$.  On
$A_n\cap B_n^{bTS}$, Lemma~\ref{lem:btsp_edge_to_loc_2} gives, for an edge
of length $\ell$ with $n^{-\alpha/d}<\epsilon\ell\le r_*$,
\[
n\ell^{p+d}\lesssim S_p\le n^{1-\alpha'p/d}.
\]
For $\epsilon\ell>r_*$, occupancy at radius $r_*$ gives
$nr_*^d\ell^p\lesssim S_p$, which is impossible for large $n$.  Hence
$L\lesssim n^{-\alpha/d}$ outside an event of arbitrarily high polynomial
decay.  Since $L\le\sqrt d$ deterministically,
\[
(\EE [L^{2(p-2)}])^{1/2}\lesssim n^{-\alpha(p-2)/d}.
\]
Inserting this in
\eqref{eq:btsp-poincare-step} leaves
$n^{p/d-1-\alpha(p-2)/d}$.  A permissible $\alpha$ makes this a negative
power exactly when $2p<d^2$, completing the proof.
\end{proof}

\appendix
\section{Numerical simulations}
\label{sec:numerics}

In this section we provide exploratory numerical evidence concerning two
questions:

\begin{itemize}
\item[(i)] whether the $p\to q$ transfer conjecture is compatible with the
observed matching and TSP costs;
\item[(ii)] whether the tested finite sizes display an obvious qualitative
change at $p=d^2/2$.
\end{itemize}

These experiments do not test the fluctuation exponent in our concentration
theorems.  Their purpose is only to probe the $p\to q$ conjecture and the
apparent absence of a singularity at the technical threshold $p=d^2/2$.

All costs are normalized by the expected scaling $r(d,q)(n)=n^{1-q/d}$ ({recall that} we work in $d\ge 3$), so that convergence to a positive constant corresponds to the conjectured behavior. The code used is freely available at the repository 
\url{https://github.com/DarioTrevisan/REMPF}.

\subsection{Numerical evidence for the \texorpdfstring{$p\to q$}{p -> q} conjecture}

We test whether, for a fixed optimiser exponent $p$, the $q$--cost of the $p$--optimal configuration scales as $n^{1-q/d}$ for $q>p$.

\medskip

\paragraph{\bf Monopartite TSP}

Figure~\ref{fig:tsp_d3_p1} shows the normalized costs for $d=3$, $p=1$ and $q=1,2,3$. 
The curves are compatible with convergence to positive constants, although
the finite-size range is too short for a reliable extrapolation.

\begin{figure}[ht!]
\centering
\begin{subfigure}{0.48\textwidth}
    \centering
    \includegraphics[width=\linewidth]{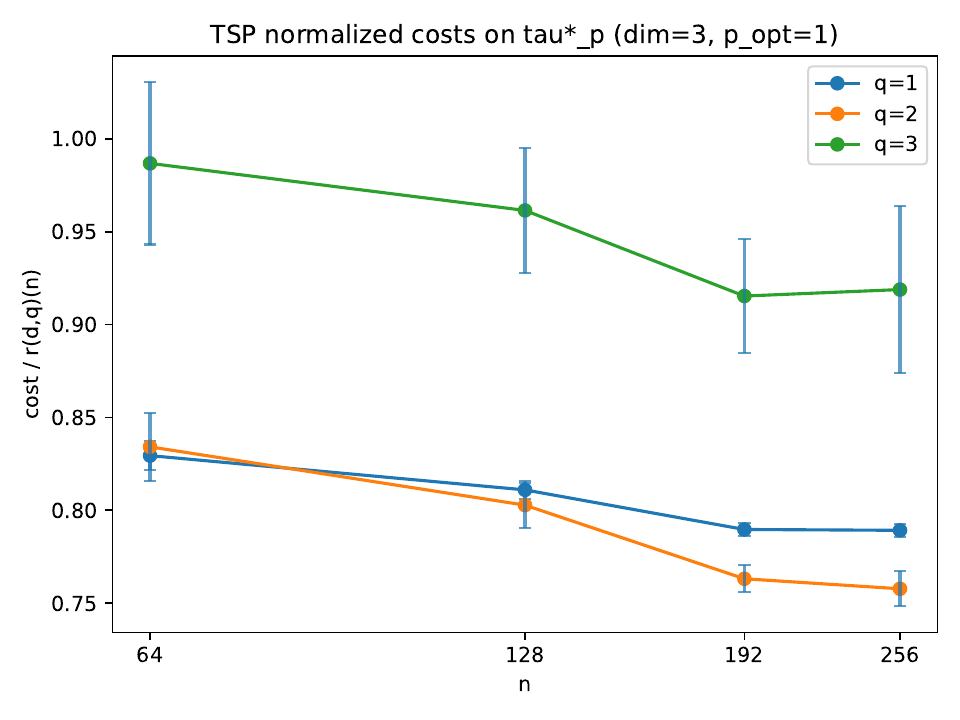}
    \caption{$d=3$, $p=1$, $q=1,2,3$.}
    \label{fig:tsp_d3_p1}
\end{subfigure}
\hfill
\begin{subfigure}{0.48\textwidth}
    \centering
    \includegraphics[width=\linewidth]{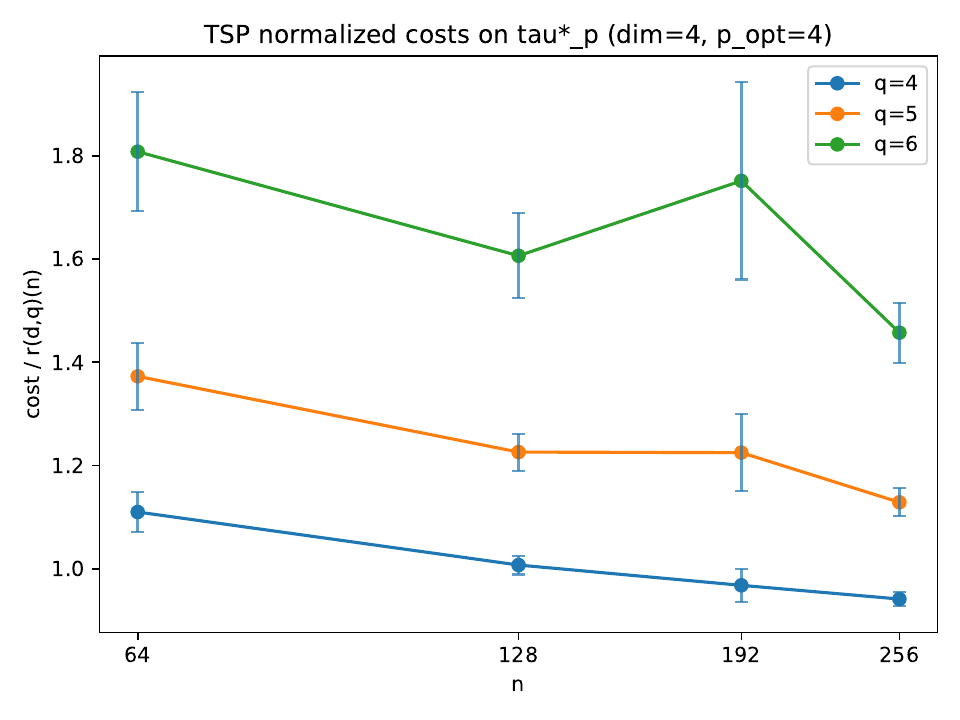}
    \caption{$d=4$, $p=4$, $q=4,5,6$.}
    \label{fig:tsp_d4_p4}
\end{subfigure}
\caption{Monopartite TSP: normalized $q$--costs on $\tau^*_p$.
In both regimes, the curves remain approximately flat after normalization by $n^{1-q/d}$.}
\label{fig:tsp_two_panel}
\end{figure}

\medskip
\paragraph{\bf Bipartite matching.}

We perform the same experiment for bipartite matching.
\begin{figure}[ht!]
\centering
\begin{subfigure}{0.48\textwidth}
    \centering
    \includegraphics[width=\linewidth]{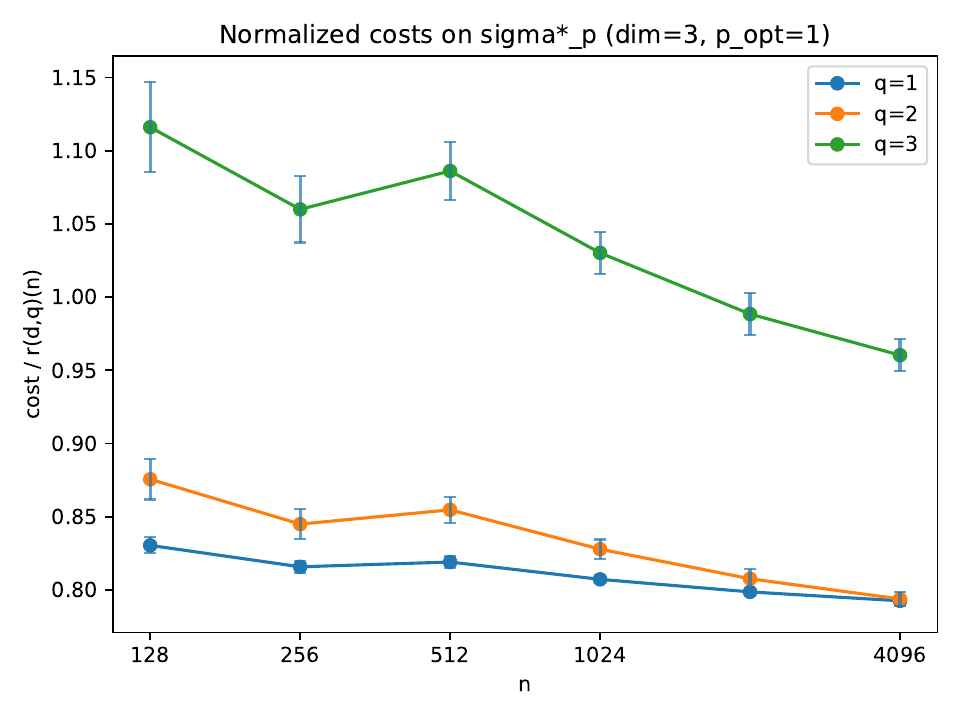}
    \caption{$d=3$, $p=1$, $q=1,2,3$.}
    \label{fig:match_d3_p1}
\end{subfigure}
\hfill
\begin{subfigure}{0.48\textwidth}
    \centering
    \includegraphics[width=\linewidth]{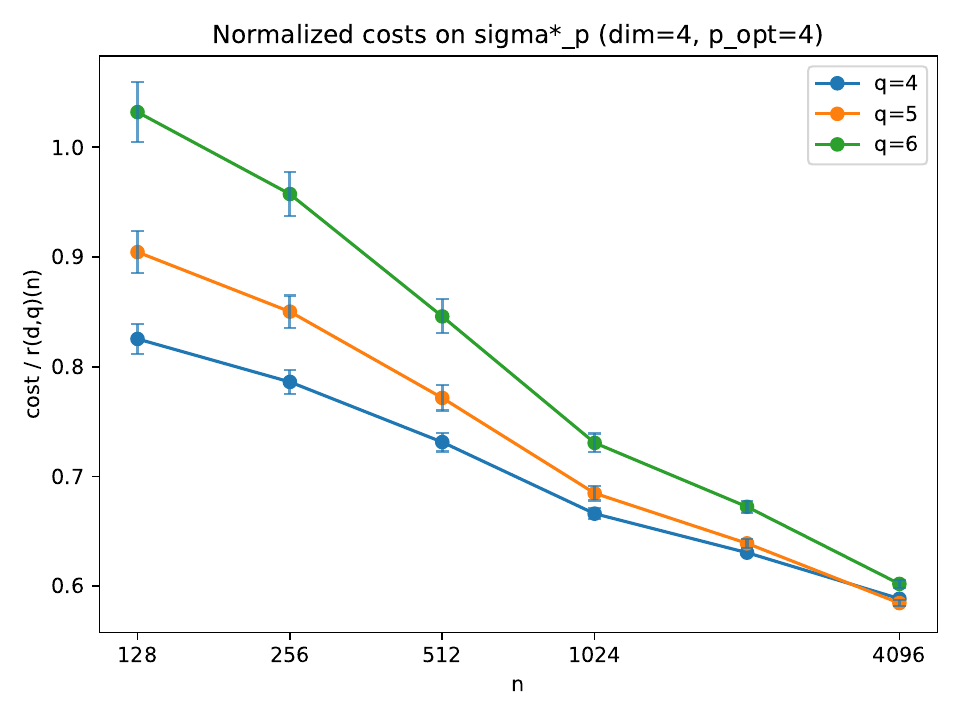}
    \caption{$d=4$, $p=4$, $q=4,5,6$.}
    \label{fig:match_d4_p4}
\end{subfigure}
\caption{Bipartite matching: normalized $q$--costs on $\sigma^*_p$.
The curves remain stable under normalization by $n^{1-q/d}$.}
\label{fig:match_two_panel}
\end{figure}

Figure~\ref{fig:match_d3_p1} corresponds to $d=3$, $p=1$, $q=1,2,3$.
The normalized costs {seems to }stabilize as $n$ grows. Figure~\ref{fig:match_d4_p4} shows the case $d=4$, $p=4$, $q=4,5,6$.
The behavior is again compatible with the conjectured $n^{1-q/d}$ scaling.

Overall, in the tested regimes the normalized $q$--costs appear bounded and
stable.  This is compatible with the $p\to q$ conjecture, {but} should be
interpreted cautiously given the available finite-size range.

\subsection{Behaviour near the threshold \texorpdfstring{$p=d^2/2$}{p = d squared/2}}
We now examine the critical case $d=4$, $p=8$, for which $p=d^2/2$.
The analytical method developed above does not yield concentration at this threshold.

\medskip
\paragraph{\bf Matching and monopartite TSP at the critical exponent.}

Figure~\ref{fig:tsp_d4_p8} shows normalized costs for $d=4$, $p=8$, $q=8,9,10$.
At the theoretical threshold, the normalized curves remain stable over the
tested sizes. Figure~\ref{fig:match_d4_p8} displays the corresponding matching
experiment.  Neither panel displays an obvious qualitative change at
$p=d^2/2$ over this finite range.

\begin{figure}[ht!]
\centering
\begin{subfigure}{0.48\textwidth}
    \centering
    \includegraphics[width=\linewidth]{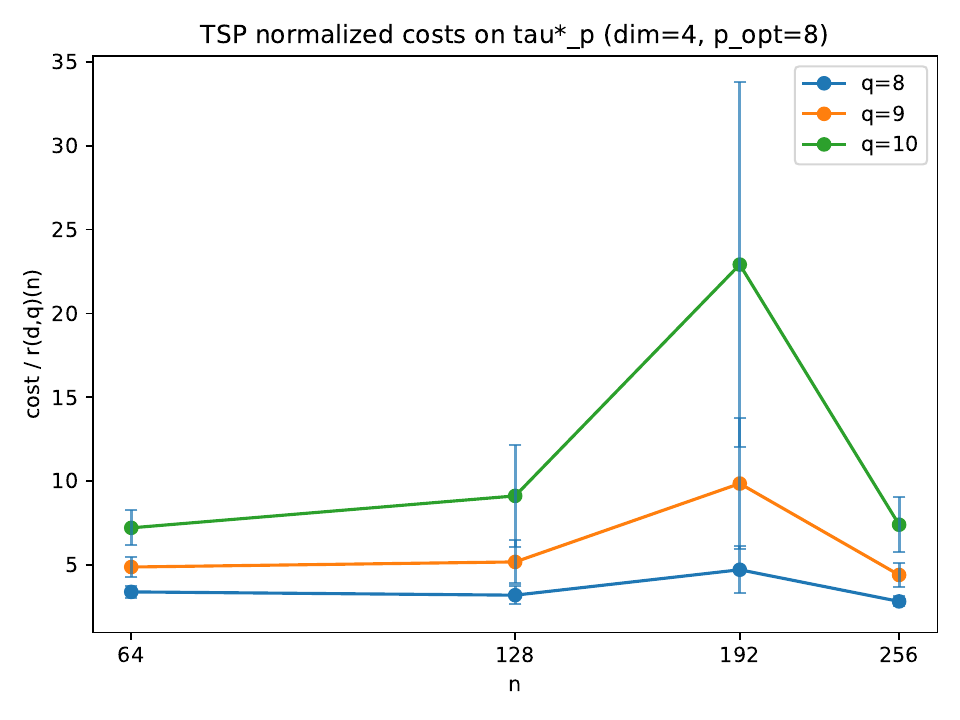}
    \caption{Monopartite TSP, $d=4$, $p=8$.}
    \label{fig:tsp_d4_p8}
\end{subfigure}
\hfill
\begin{subfigure}{0.48\textwidth}
    \centering
    \includegraphics[width=\linewidth]{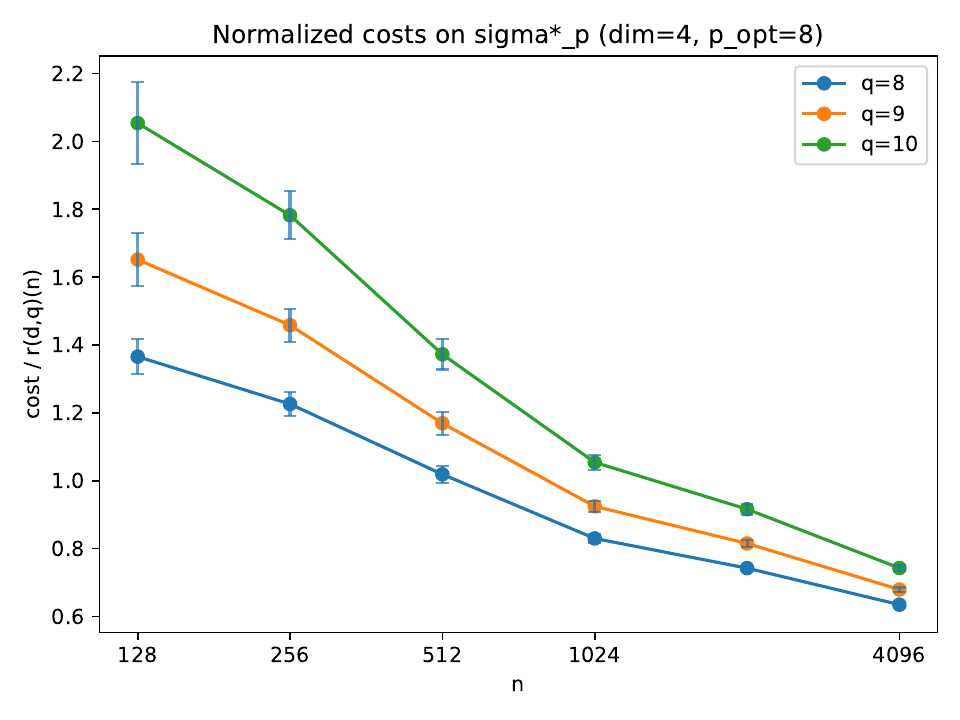}
    \caption{Bipartite matching, $d=4$, $p=8$.}
    \label{fig:match_d4_p8}
\end{subfigure}
\caption{Critical regime $p=d^2/2$ ($d=4$, $p=8$).
No obvious qualitative change is visible over the tested sizes.}
\label{fig:critical_two_panel}
\end{figure}

These observations are compatible with the restriction $p<d^2/2$ being a
limitation of the present maximal-edge/Poincar\'e argument.  They do not rule
out other behaviour outside the tested sizes and do not constitute evidence
for a fluctuation exponent.

\begin{center}
  \FundingLogos
  
  \vspace{0.5em}
  \begin{tcolorbox}\centering\small
   
    Funded by the European Union. Views and opinions expressed are however those of the author(s) only and do not necessarily reflect those of the European Union or the European Research Council Executive Agency. Neither the European Union nor the granting authority can be held responsible for them. This project has received funding from the European Research Council (ERC) under the European Union’s Horizon Europe research and innovation programme (grant agreement No. 101198055, project acronym NEITALG).
    
  \end{tcolorbox}
\end{center}

\subsection*{Conflict of Interest} The authors declare no conflict of interest.

\bibliographystyle{abbrv}
\bibliography{OT}

@book{HardyLittlewoodPolya1952,
  author    = {Hardy, G. H. and Littlewood, J. E. and P{\'o}lya, G.},
  title     = {Inequalities},
  edition   = {2},
  publisher = {Cambridge University Press},
  address   = {Cambridge},
  year      = {1952}
}

@article{AmGlaTre,
	author = {Ambrosio, Luigi and Glaudo, Federico and Trevisan, Dario},
	doi = {10.3934/dcds.2019304},
	fjournal = {Discrete and Continuous Dynamical Systems. Series A},
	issn = {1078-0947},
	journal = {Discrete Contin. Dyn. Syst.},
	mrclass = {60D05 (35B35 35F21 49J55 49Q20 58J35)},
	mrnumber = {4026190},
	number = {12},
	pages = {7291--7308},
	title = {On the optimal map in the 2-dimensional random matching problem},
	url = {https://doi-org.accesdistant.sorbonne-universite.fr/10.3934/dcds.2019304},
	volume = {39},
	year = {2019}}

@article{AKT84,
	author = {M. {Ajtai} and J. {Koml\'os} and G. {Tusn\'ady}},
	doi = {10.1007/BF02579135},
	fjournal = {{Combinatorica}},
	issn = {0209-9683; 1439-6912/e},
	journal = {{Combinatorica}},
	msc2010 = {60D05 60G40},
	pages = {259--264},
	publisher = {Springer, Berlin/Heidelberg; J\'anos Bolyai Mathematical Society, Budapest},
	title = {{On optimal matchings.}},
	volume = {4},
	year = {1984},
	zbl = {0562.60012}}

@article{FoGu15,
	author = {Fournier, N. and Guillin, A.},
	doi = {10.1007/s00440-014-0583-7},
	fjournal = {Probability Theory and Related Fields},
	issn = {0178-8051},
	journal = {Probab. Theory Related Fields},
	mrclass = {60F25 (60E15 60F10)},
	mrreviewer = {Jos\'e Trashorras},
	number = {3-4},
	pages = {707--738},
	title = {On the rate of convergence in {W}asserstein distance of the empirical measure},
	url = {https://doi.org/10.1007/s00440-014-0583-7},
	volume = {162},
	year = {2015}}

@article{AmStTr16,
	author = {Ambrosio, L. and Stra, F. and Trevisan, D.},
	fjournal = {Probabability Theory and Related Fields},
	journal = {{Probab. Theory Relat. Fields}},
	number = {1-2},
	pages = {433--477},
	publisher = {Springer},
	title = {A {PDE} approach to a 2-dimensional matching problem},
	volume = {173},
	year = {2019}}

@book{VilOandN,
	author = {Villani, C.},
	doi = {10.1007/978-3-540-71050-9},
	isbn = {978-3-540-71049-3},
	pages = {xxii+973},
	publisher = {Springer-Verlag, Berlin},
	series = {Grundlehren der Mathematischen Wissenschaften},
	title = {Optimal transport: old and new},
	url = {http://dx.doi.org/10.1007/978-3-540-71050-9},
	volume = {338},
	year = {2009}}

@article{CaLuPaSi14,
	author = {Caracciolo, S. and Lucibello, C. and Parisi, G. and Sicuro, G.},
	journal = {Physical Review E},
	number = {1},
	publisher = {APS},
	title = {Scaling hypothesis for the {E}uclidean bipartite matching problem},
	volume = {90},
	year = {2014}}

@article{GO,
	title     = {A variational proof of partial regularity for optimal
               transportation maps},
  author    = {Michael Goldman and Felix Otto},
  journal   = {Ann. Sci. Ec. Norm. Super. (4)},
  publisher = {Soci\'et\'e Math\'ematique de France},
  volume    =  {53},
  number    =  {5},
  pages     = {1209--1233},
  year      =  {2020}
}

@article{GoldTrev24ran,
  author = {Michael Goldman and Dario Trevisan},
  title = {Optimal transport methods for combinatorial optimization over two random point sets},
  journal = {Probability Theory and Related Fields},
  year = {2024},
  volume = {188},
  number = {3},
  pages = {1315--1384},
  month = apr,
  issn = {1432-2064},
  doi = {10.1007/s00440-023-01245-1},
  url = {https://doi.org/10.1007/s00440-023-01245-1}
}

@article{caracciolo2018solution,
  title={Solution for a bipartite Euclidean traveling-salesman problem in one dimension},
  author={Caracciolo, Sergio and Di Gioacchino, Andrea and Gherardi, Marco and Malatesta, Enrico M},
  journal={Physical Review E},
  volume={97},
  number={5},
  pages={052109},
  year={2018},
  publisher={APS}
}

@article{capelli2018exact,
  title={Exact value for the average optimal cost of the bipartite traveling salesman and two-factor problems in two dimensions},
  author={Capelli, Riccardo and Caracciolo, Sergio and Di Gioacchino, Andrea and Malatesta, Enrico M},
  journal={Physical Review E},
  volume={98},
  number={3},
  pages={030101},
  year={2018},
  publisher={APS}
}

@book{BGL14,
  author    = {Bakry, Dominique and Gentil, Ivan and Ledoux, Michel},
  title     = {Analysis and Geometry of {M}arkov Diffusion Operators},
  publisher = {Springer},
  series    = {Grundlehren der mathematischen Wissenschaften},
  volume    = {348},
  year      = {2014},
  doi       = {10.1007/978-3-319-00227-9}
}

@article{Ta92,
  author  = {Talagrand, Michel},
  title   = {Matching Random Samples in Many Dimensions},
  journal = {The Annals of Applied Probability},
  year    = {1992},
  volume  = {2},
  number  = {4},
  pages   = {846--856},
  doi     = {10.1214/aoap/1177005578}
}

@article{BoLe21,
  author  = {Bobkov, Sergey G. and Ledoux, Michel},
  title   = {A simple {F}ourier analytic proof of the {A}jtai--{K}oml{\'o}s--{T}usn{\'a}dy optimal matching theorem},
  journal = {The Annals of Applied Probability},
  year    = {2021},
  volume  = {31},
  number  = {6},
  pages   = {2567--2584},
  doi     = {10.1214/20-AAP1656}
}

@article{dBGM99,
  author  = {del Barrio, Eustasio and Gin{\'e}, Evarist and Matr{\'a}n, Carlos},
  title   = {Central limit theorems for the {W}asserstein distance between the empirical and the true distributions},
  journal = {The Annals of Probability},
  year    = {1999},
  volume  = {27},
  number  = {2},
  pages   = {1009--1071},
  doi     = {10.1214/aop/1022677394}
}

@book{PenroseRGG03,
  author    = {Penrose, Mathew},
  title     = {Random Geometric Graphs},
  publisher = {Oxford University Press},
  year      = {2003},
  series    = {Oxford Studies in Probability},
  doi       = {10.1093/acprof:oso/9780198506263.001.0001}
}

@book{DevroyeLugosi01,
  author    = {Devroye, Luc and Lugosi, G{\'a}bor},
  title     = {Combinatorial Methods in Density Estimation},
  publisher = {Springer},
  year      = {2001},
  series    = {Springer Series in Statistics},
  doi       = {10.1007/978-1-4613-0121-9}
}

@book{Steele97,
  author    = {Steele, J. Michael},
  title     = {Probability Theory and Combinatorial Optimization},
  publisher = {SIAM},
  year      = {1997},
  series    = {CBMS-NSF Regional Conference Series in Applied Mathematics},
  doi       = {10.1137/1.9781611970069}
}

@book{Yukich98,
  author    = {Yukich, Joseph E.},
  title     = {Probability Theory of Classical Euclidean Optimization Problems},
  publisher = {Springer},
  year      = {1998},
  series    = {Lecture Notes in Mathematics},
  volume    = {1675},
  doi       = {10.1007/BFb0093421}
}

@article{BHH59,
  author  = {Beardwood, J. and Halton, J. H. and Hammersley, J. M.},
  title   = {The shortest path through many points},
  journal = {Proceedings of the Cambridge Philosophical Society},
  year    = {1959},
  volume  = {55},
  number  = {4},
  pages   = {299--327},
  doi     = {10.1017/S0305004100034095}
}

@article{Arora98,
  author  = {Arora, Sanjeev},
  title   = {Polynomial time approximation schemes for {E}uclidean traveling salesman and other geometric problems},
  journal = {Journal of the ACM},
  year    = {1998},
  volume  = {45},
  number  = {5},
  pages   = {753--782},
  doi     = {10.1145/290179.290180}
}

@article{Mitchell99,
  author  = {Mitchell, Joseph S. B.},
  title   = {Guillotine subdivisions approximate polygonal subdivisions: A simple {PTAS} for geometric {TSP}, $k$-{MST}, and related problems},
  journal = {SIAM Journal on Computing},
  year    = {1999},
  volume  = {28},
  number  = {4},
  pages   = {1298--1309},
  doi     = {10.1137/S0097539796309764}
}

@article{LinKernighan73,
  author  = {Lin, Shen and Kernighan, Brian W.},
  title   = {An effective heuristic algorithm for the traveling-salesman problem},
  journal = {Operations Research},
  year    = {1973},
  volume  = {21},
  number  = {2},
  pages   = {498--516},
  doi     = {10.1287/opre.21.2.498}
}

@article{Croes58,
  author  = {Croes, Georges A.},
  title   = {A method for solving traveling-salesman problems},
  journal = {Operations Research},
  year    = {1958},
  volume  = {6},
  number  = {6},
  pages   = {791--812},
  doi     = {10.1287/opre.6.6.791}
}

@article{MezardParisi88,
  author  = {M{\'e}zard, Marc and Parisi, Giorgio},
  title   = {The {E}uclidean matching problem},
  journal = {Journal de Physique},
  volume  = {49},
  pages   = {2019--2025},
  year    = {1988}
}

@book{Sicuro17thesis,
  author    = {Sicuro, Gabriele},
  title     = {The {E}uclidean Matching Problem},
  series    = {Springer Theses},
  publisher = {Springer},
  address   = {Cham},
  doi       = {10.1007/978-3-319-46577-7},
  year      = {2017}
}

@article{Aldous01,
  author  = {Aldous, David J.},
  title   = {The $\zeta(2)$ limit in the random assignment problem},
  journal = {Random Structures \& Algorithms},
  volume  = {18},
  number  = {4},
  pages   = {381--418},
  year    = {2001}
}

\end{document}